\documentclass[12pt]{amsart}

\usepackage{graphicx}
\usepackage{amsthm}
\usepackage{amsmath, amssymb}
\swapnumbers

\theoremstyle{plain}
\newtheorem{Thm}{Theorem}

\newtheorem{Coro}[Thm]{Corollary}
\newtheorem{Lem}[Thm]{Lemma}
\newtheorem{Claim}[Thm]{Claim}
\newtheorem{Prop}[Thm]{Proposition}

\theoremstyle{definition}

\begin{document}

\title{The coarse geometry of the Kakimizu complex}

\author{Jesse Johnson}
\address{\hskip-\parindent
        Department of Mathematics \\
        Oklahoma State University \\
        Stillwater, OK 74078 \\
        USA}
\email{jjohnson@math.okstate.edu}
\author{Roberto Pelayo}
\address{\hskip-\parindent
	Department of Mathematics \\
	University of Hawaii at Hilo \\
	Hilo, HI 96720 \\
	USA}
\email{robertop@hawaii.edu}
\author{Robin Wilson}
\address{\hskip-\parindent
	Department of Mathematics and Statistics\\	
	California State Polytechnic University, Pomona \\
	Pomona, CA 91768 \\
	USA}
\email{robinwilson@csupomona.edu}


\thanks{The first author was supported by NSF Grant DMS-1006369}
\thanks{The second author was supported by NSF Grant DMS-1045147}

\begin{abstract}
We show that the Kakimizu complex of minimal genus Seifert surfaces for a knot in the 3-sphere is quasi-isometric to a Euclidean integer lattice $\mathbb Z^n$ for some $n \geq 0$.  
\end{abstract}

\maketitle

\section{Introduction}

In general, a knot $K \subset S^3$ may have multiple non-isotopic minimal genus Seifert surfaces.  To understand all these possibilities, Kakimizu \cite{K} defined a simplicial complex $\mathcal{MS}(K)$, later referred to as the \textit{Kakimizu complex}.  Each vertex $\sigma$ of  $\mathcal{MS}(K)$ is an isotopy class of minimal genus Seifert surfaces for $K$, and $n$-simplices are spanned by isotopy classes with pairwise disjoint Seifert surface representatives.  The metric on $\mathcal{MS}(K)$ is defined by the minimal lengths of edge paths between vertices. Kakimizu~\cite{K} defined a metric on the complex using the infinite cyclic cover of $K$ and showed that this metric is equal to the edge path metric.

Recently,  $\mathcal{MS}(K)$ has been shown to be connected~\cite{S-T}, simply connected, and contractible~\cite{P-S}. In fact, $\mathcal{MS}(K)$ for several classes of knots has been computed, including special arborescent knots~\cite{S} and certain composite knots~\cite{K}.  Furthermore, the Kakimizu complex has been computed for all prime knots up to $10$ crossings~\cite{K2}.

For hyperbolic knots, Pelayo \cite{P} and Sakuma-Schackleton \cite {S-S} give a bound on the diameter of the Kakimizu complex that is quadratic in the genus of the knot, and Wilson \cite{W} shows that $\mathcal{MS}(K)$ is finite.  For satellite knots, however, $\mathcal{MS}(K)$ may be infinite, and may even be locally infinite \cite{Ba}.  

The goal of this paper is to describe the coarse geometry of the Kakimizu complex. Recall that a \emph{quasi-isometry} is a map $f: X \rightarrow Y$ between metric spaces $X, Y$ such that $\frac{1}{L} d_Y(f(x),f(y)) - L \leq d_X(x,y) \leq Ld_Y(f(x),f(y)) + L$ for some constant $L$ and every point of $Y$ is within an $L$-neighborhood of the image $f(X)$. For torus and hyperbolic knots, $\mathcal{MS}(K)$ is finite and therefore quasi-isometric to a single point.  For satellite knots, the large-scale structure may be more exciting.

\begin{Thm}
\label{mainthm}
For any knot $K \subset S^3$, the Kakimizu complex $\mathcal{MS}(K)$ is quasi-isometric to $\mathbb Z^n$ for some $n \geq 0$.
\end{Thm}

The dimension $n$ of of the abelian group can be calculated in a relatively straightforward fashion. Below, we define a subset of the complementary pieces in the JSJ decomposition for the complement of $K$ called the core. It follows from the proof that the value of $n$ is equal to the number of  JSJ tori in the interior of the core minus the number of fibered complementary components in the core.

The outline of the paper is as follows: In Section~\ref{jsjsect}, we examine how Seifert surfaces for $K$ interact with the incompressible tori in a JSJ decomposition for the knot complement. In Section~\ref{gractionsect}, we define a group action on $\mathcal{MS}(K)$ by an abelian group, generated by twisting around the tori in the JSJ decomposition, then in Section~\ref{mainproofsect}, we prove that this action induces a quasi-isometry from $\mathcal{MS}(K)$ to $\mathbb{Z}^n$.

Note that the proof here is for knots in the 3-sphere, rather than links. Przytycki-Schultens~\cite{P-S} discuss ways of generalizing the Kakimizu complex to manifolds with multiple boundary components, but our proof relies on certain properties that are unique to knots, in particular the fact (proved below as a Corollary of a result proved by Ryan Budney~\cite{Bu}) that the Seifert fibered components of the complement of a JSJ decomposition are all small.  However, there are no known counterexamples such a generalization of our theorem for links.

\section{The Kakimizu Complex and the JSJ decomposition}
\label{jsjsect}

In \cite{K}, Kakimizu computes the Kakimizu complex for the connected sum of two non-fibered knots $K_1$ and $K_2$ with unique incompressible Seifert surfaces.  In this case, $\mathcal{MS}(K_1 \# K_2)$ is isometric to $\mathbb Z \subset \mathbb R$.  These Seifert surfaces come from taking the canonical Seifert surface obtained by forming the connected sum of the minimal genus Seifert surfaces for each knot and spinning it around the incompressible swallow-follow torus in the complement of the composite knot.  When a knot has more than two factors, more incompressible tori would mean more ways to potentially create new Seifert surfaces by spinning.  

To understand this structure, let $M_K$ be the knot complement, and consider the JSJ decomposition of $M_K$: let $T_1,\dots, T_n$ be a minimal collection of pairwise disjoint, incompressible tori such that the complement of $\bigcup T_i$ consists of Seifert fibered pieces and atoroidal (hyperbolic) pieces.  Each $T_i$ bounds a solid torus in $S^3$ containing $K$ on one side; we will transversely orient each $T_i$ so that the knot is on the negative side.   If we consider a neighborhood $\mathcal N(T_i)$ of each torus and take the complement of the interior of these neighborhoods in $M_K$, then $M_K - \bigcup_{i=1}^n int(\mathcal N(T_i))$ is a collection of compact connected components that we will call \emph{blocks}.

Let $B$ be a block not containing the knot and let $M$ be the component of the complement $M_K \setminus B$ that contains $K$.  If the knot is not nullhomologuous in $M$, then we will say that $B$ is a \emph{core block}.  This implies that every minimal genus Seifert surface for $K$ must intersect the torus $\partial M$ and therefore the interior the core block $B$.  Therefore, every minimal genus Seifert surface for $K$ must intersect every core block.  If the block $B$ contains the knot, then we also define it to be a core block.  We refer to the union of all the core blocks as the \emph{core} of the JSJ decomposition of $M_K$.  The reader can check that the core is a connected subset of $M_K$.  

One consequence of the definition of the core is that minimal genus Seifert surfaces must intersect tori in the interior of the core in a very controlled manner.

\begin{Prop}  \label{fixedslope} Let T be a JSJ torus in the interior of the core.  There is a fixed slope $\alpha$ of $T$ such that every Seifert surface $S$ for the knot must intersect $T$ in one or more parallel loops with exactly this slope.
\end{Prop}

\begin{proof}  Since $T$ is a JSJ torus, it separates $S^3$ into two components $M$ and $M'$ where $M$ is a solid torus containing the knot and $M'$ is on the opposite side.  Since $T$ is in the interior of the core, $S$ must intersect $T$ non-trivially. Then $S \cap T$ is a collection of parallel loops with orientations induced from $S$ and $T$ that define non-trivial element of the first homology group of $T$.  On the other side, $M'$ is homeomorpgic to a knot complement for some knot $K'$ in $S^3$.  The surface $S' = S \cap M'$ implies that the loops $S \cap T$ determine a trivial element of the first homology of $M'$, so $\partial S' = S \cap T$ must be a collection of parallel copies of a longitude $\alpha \subset T$ for the knot $K'$, which generates the kernel of the inclusion map $H_1(T) \rightarrow H_1(M')$.  For any other Seifert surface $F$ for the knot $K$ intersecting $T$, we must have that $F \cap T$ is homologous to $S \cap T$, since both of these collections of intersection curves are homologous to $K$ in $M$. Therefore, $F$ and $S$ must intersect $T$ in the same slope $\alpha$.  
\end{proof} 

\begin{Lem} 
\label{bound}
Let $S$ be a minimal genus Seifert surface for $K$ and let $T = \bigcup_{i=1}^n T_i$ be the collection of all JSJ tori.  If $S$ is isotoped to intersect $T$ minimally, then $\left| S\cap T\right| \leq 6g-4$, where $g$ is the genus of $S$.
\end{Lem}

\begin{proof}
Let $B_1,...,B_k$ denote the blocks of the JSJ decomposition.  Each $B_i$ is a submanifold of $M_K$ that is either a hyperbolic link complement and hence is atoroidal, or Seifert fibered.  Notice that $S$ meets each $B_i$ in a collection of orientable, disjoint, essential surfaces that are properly embedded in $B_i$.  

Since $S$ is a once-punctured surface of genus $g$, there are at most $3g-2$ isotopy classes of essential loops in $S$.  Consider the components of $S \cap T$, each of which is a simple closed curve.  Since both $S$ and $T$ are incompressible and the complement of $K$ is irreducible, we can assume that all intersection curves are essential in both surfaces. Otherwise we could reduce the number of intersections, contradicting the minimality of $S \cap T$.  In particular, all curves of $S \cap T_i$ for each $i$ are essential in $S$.  Therefore, there are no more than $3g-2$ isotopy classes of curves in $S$ for each intersection $S \cap T_i$. 

In order to bound the number of components in $S\cap T$, it suffices to bound the number of parallel pairwise disjoint curves of $S\cap T_i$ for each $i$.  Disjoint curves in $S \cap T$ that are parallel in $S$ cobound an annulus $A$ in $S$, and this annulus must be incompressible and properly embedded in some block $B_i$.  If $A$ were boundary parallel then we could reduce $S \cap T$. Thus $A$ is an essential annulus so $B_i$ must be Seifert fibered with $A$ isotopic to a union of fibers.

Assume for contradiction there are three adjacent pairwise disjoint curves in $S\cap T_i$ that are parallel in $S$.  Then the three curves correspond to two adjacent essential annuli $A_1$ and $A_2$ contained in adjacent Seifert fibered blocks.  Without loss of generality, we can assume that $A_1$ is properly embedded in block $B_1$ and $A_2$ is properly embedded in block $B_2$.  Because each $A_i$ is a union of fibers, the two fiberings of the common boundary torus $T_i$ induced from the Seifert fiberings of $B_1$ and $B_2$ have the same slope.  Therefore these two fiberings can be isotoped to agree on $T_i$ (see~\cite{H}), contradicting the minimality of the JSJ decomposition since $T_i$ can be removed from the collection of JSJ tori.  Therefore, there can be at most two adjacent curves of $S\cap T$ that are parallel in $S$ in each isotopy class of curves.  Hence $|S \cap T| \leq 2(3g-2)=6g-4$.  \end{proof}

The following is a slight generalization of the main result in \cite{W}. The proof  can be found in \cite{W}, however the statement is for manifolds with one toroidal boundary component.  It is not difficult to modify the proof to also hold for manifolds with a finite number of toroidal boundary components, by a minor modification of the normal surface equations.  

\begin{Thm}[\cite{W}]
\label{Wilson}
Let $M_L$ be a link complement.  Let $\alpha_1, \ldots, \alpha_k$ be a set of preferred longitudes for the link $L$.  If $M_L$ contains an infinite collection of essential surfaces $S_i$ of the same Euler characteristic such that $\partial S_i$ is isotopic to a subcollection of the $\alpha_i$ and there exists an $N$ such that $|\partial S_i| \leq N < \infty$ for each $i$, then $M_L$ contains a closed incompressible torus.
\end{Thm}

The following Corollary follows immediately from Theorem~\ref{Wilson}.

\begin{Coro}
\label{WilsonCoro}
Let $M_L$ be a link complement and $N \in \mathbf{N}$.  Suppose that $M_L$ contains no closed essential tori.  Let $\alpha_1, \ldots, \alpha_k$ be a set of preferred longitudes for the link $L$.  Then $M_L$ contains at most finitely many essential surfaces $S_i$ of maximal Euler characteristic such that $\partial S$ is isotopic to a subcollection of the $\alpha_i$ and $|\partial S_i| \leq N$ for each $i$.
\end{Coro}

A Seifert fibered block of a JSJ decomposition may be toroidal, but the Seifert-fibered blocks of a knot complement in $S^3$ are much more restricted. Proposition 3.2 of~\cite{Bu} gives the following classification of Seifert-fibered submanifolds of $S^3$.

\begin{Lem}[\cite{Bu}]
\label{SFsub}  Let $V \neq S^3$ be a Seifert-fibered sub-manifold of $S^3$, then $V$ is diffeomorphic to one of the following:

\begin{itemize}

\item[$\cdot$] A Seifert-fibered space over an $n$-times puncture sphere with two exceptional fibers, appearing as the complement of $n$ regular fibers in a Seifert fibering of $S^3$.

\item[$\cdot$] A Seifert-fibered space over an $n$-times punctured sphere with $1$ exceptional fiber, appearing as the complement of $n-1$ regular fibers in a Seifert fibering of an embedded solid torus in $S^3$.  

\item[$\cdot$] A Seifert-fibered space over an $n$-times punctured sphere with no exceptional fibers.

\end{itemize}
\end{Lem}

We will use the above Lemma to show that there are only finitely many incompressible surfaces in a Seifert fibered block for a knot complement.

\begin{Thm}
\label{fin surfaces}
Let $K$ be a knot $K$ and let $B$ be a core block of the JSJ decomposition for $M_K$.  There exist finitely many essential surfaces $S_1, S_2, \ldots, S_m$ such that for any minimal genus Seifert surface $S$ for $K$, every component of $S \cap B$ is isotopic to one of the $S_i$.  
\end{Thm}

\begin{proof}  If $B$ is a hyperbolic block, then it is atoroidal and $\partial (S \cap B)$ is fixed and bounded by Proposition~\ref{fixedslope} and Lemma~\ref{bound}.  In this case, the conclusion follows directly from  Corollary~\ref{WilsonCoro}.  If $B$ is a Seifert-fibered block, then Lemma~\ref{SFsub} gives us three possibilities.  In the first case, $B$ is Seifert fibered over an $n$-times punctured sphere with two singular fibers, arising as the complement of $n$ solid tori.  Since $M_K$ is a knot complement, there can only be one such solid torus in the complement of $B$.  To see this, note that any solid torus contains a compressing disk for its boundary.  Since the boundary tori of $B$ are incompressible in the knot complement, the knot must intersect this compressing disk.  However, since there is only one knot, there exists only one solid torus in the complement of the block. Thus, $n = 1$ and $B$ is Seifert-fibered over a once-punctured sphere with two critical fibers.  Such Seifert-fibered spaces are known to contain no closed essential surfaces and thus are atoroidal.  Applying Corollary~\ref{WilsonCoro}, the conclusion again follows.  

In the second case, $B$ is Seifert fibered over an $n$-punctured sphere with one critical fiber, arising as the complement of $n-1$ solid tori.  Once again, as $B$ is a subset of a knot complement, there can be at most $1$ solid torus, so $n \leq 2$.  These Seifert-fibered spaces are also atoroidal and thus applying Corollary~\ref{WilsonCoro}, the conclusion follows.  In the last case, $B$ is Seifert-fibered over an $n$-times punctured sphere with no exceptional fibers and is thus a product.  Since $B$ is a core block, by Proposition~\ref{fixedslope}, the boundary of $S \cap B$ is specified.  In such product spaces, there is a unique incompressible surface with specified boundary.
\end{proof}

\section{Group actions on the Kakimizu complex}
\label{gractionsect}

Every automorphism of the knot complement induces an automorphism of the Kakimizu complex.   For a given $T_i$, let $U$ be a closed regular neighborhood homeomorphic to $I \times T_i$. We will define an automorphism of the knot complement that is the identity outside of $U$ and spins around the JSJ torus in a given direction.  Consider the universal cover of $T_i$, which is homeomorphic to the plane.  Choose a coordinate system on this plane.  For every integer vector $(m,n) \in \mathbb Z \times \mathbb Z$, there is a family of automorphisms $\phi_t$ of $T_i$ for $t \in I$ that lift to translations of the plane in the direction of $(m,n)$ and such that $\phi_0$ and $\phi_1$ are the identity on $T_i$.  We obtain an automorphism of the knot complement as follows:  $$\Phi(x) = \left\{ \begin{array} {ll} (t , \phi_t(z)) & \textrm{ if } \, x = (t,z) \in I \times T_i \\ x & \textrm{else} \end{array} \right. .$$
 
Notice that for a fixed $T_i$,  each choice of integer vector $(m,n) \in \mathbb Z \times \mathbb Z$ gives one of these automorphisms of the knot complement.  Furthermore, composition of these automorphisms corresponds to integer vector addition in $\mathbb Z \times \mathbb Z$, which forms an abelian group of rank $2$.  A rank one subgroup of $\mathbb{Z} \times \mathbb{Z}$ acts trivially on the Kakimizu complex.  To see this, let $\alpha$ be a basis element for one of the factors of $\mathbb{Z}$ representing the slope on T given by Proposition~\ref{fixedslope}.  Spinning parallel to $\alpha$ induces an action of one of the $\mathbb{Z}$ factors on the Kakimizu complex.  This spinning leaves the isotopy classes of minimal genus Seifert surfaces fixed so its corresponding action on the Kakimizu complex is trivial. For each $T_i$, let $\Phi_i$ be the automorphism orthogonal to this trivial one. \\

Consider $G'$, the group of all automorphisms of the knot complement generated by these homeomorphisms $\Phi_i$.  For $i \neq j$, the support of $\Phi_i$ is disjoint from that of $\Phi_j$, so such homeomorphisms commute and $G'$ is abelian.  Let $N$ be the (normal) subgroup of $G'$ that acts trivially on the Kakimizu complex.  Then $G = G'/N,$ is also a finitely generated abelian group.  

To prove Theorem~\ref{mainthm},  we will first prove that the result holds for a particular subcomplex of $\mathcal{MS}(K)$ that is locally finite  ($MS(K)$ is not locally finite in general \cite{Ba}.)   We define the \emph{core Kakimizu complex} $\mathcal MS(C)$ to be subcomplex spanned by Seifert surfaces that can be isotoped into $C$. Because each $\Phi_i$ takes each block onto itself, the action of $G$ restricts to an action on $\mathcal{MS}(C)$.  

First, we prove some properties about the core Kamimizu complex.

\begin{Claim}  $\mathcal{MS}(C)$ is non-empty.
\label{nonempty}

\end{Claim}

\begin{proof}
Let $S$ be a Seifert surface for the knot $K$ such that $S \cap (\bigcup T_i)$ is minimal over all Seifert surfaces for $K$.  If $S$ stays inside the core, then $MS(C)$ is non-empty.  If $S$ exits the core, then it must do so by intersecting some JSJ torus $T$ that separates a core block from a non-core block.   Assume that $S$ intersects $T$ minimally.   Since $T$ is not in the interior of the core, the knot is homology trivial in the component of the complement of T that contains the knot.  Then, the intersection curves of $S$ with $T$ are nulhomologous, thus these curves of intersection occur in pairs with opposite orientations.  

Choose a pair of adjacent curves of intersection $\beta$ and $\gamma$ with opposite orientations.  Then $\beta$ and $\gamma$ co-bound an annulus $A \subset T$ with interior disjoint from $S$. Cut the Seifert surface $S$ along the curves $\beta$ and $\gamma$ and attach the resulting boundary components to $\partial A$, then push the resulting surface slightly into the interior of the block, reducing the number of intersections of $S$ with $T$ and thus contradiction the assumption that $S \cap (\bigcup T_i)$ is minimal.  Therefore the vertex representing $S$ is in $\mathcal{MS}(C)$, so $\mathcal{MS}(C)\neq \emptyset$.   
\end{proof}

\begin{Claim}
\label{connected}
$\mathcal{MS}(C)$ is connected.
\end{Claim}

\begin{proof}
To show that the subcomplex $\mathcal{MS}(C)$ is connected, we examine the construction used by Scharlemann-Thompson \cite{S-T} to find a path in the Kakimizu complex between any vertices $v_1$ and $v_2$.  Let $S$ and $S'$ be Seifert surfaces representing the vertices $v_1$ and $v_2$, respectively.   By taking double curve sums, Scharlemann-Thompson create a sequence of minimal genus Seifert surfaces $S_i$ for $0 \leq i \leq k$ such that $S_i \cap S_{i+1} = \emptyset$ for $0 \leq i < k$, with $S_0$ isotopic to $S$ and $S_k$ isotopic to $S'$.  If $S$ and $S'$ are both in $C$ then so are all the double curve sums. Thus, $S_i \subset C$ for each $i$ and the path is contained in $\mathcal{MS}(C) \subset \mathcal{MS}(K).$ 
\end{proof}

\begin{Claim}
\label{locallyfinite}
$\mathcal{MS}(C)$ is locally finite. 
\end{Claim}

\begin{proof}
Let $S$ be a minimal genus Seifert surface representing a vertex $v \in \mathcal{MS}(C)$. For any minimal genus Seifert surface $S' \subset C$ disjoint from $S$, there are finitely many possibilities for the intersection of $S'$ with each core block of the JSJ decomposition by Theorem~\ref{fin surfaces}. The surface $S'$ is determined by these intersections and the annuli that connect these subsurfaces inside the regular neighborhoods of the tori of $T$ in the interior of the core. Because $S$ intersects every component of $T$ in the interior of the core and $S'$ is disjoint from $S$, there are finitely many ways the subsurfaces can be connected together (In particular, no annulus can spin all the way around such a torus without crashing through $S$.) so there are finitely many minimal genus Seifert surfaces (up to isotopy) disjoint from $S$.\end{proof}

In order to prove that our Main Theorem holds for the subcomplex $\mathcal{MS}(C)$, we will need Theorem 25 of \cite{deLa}, which is stated below.

\begin{Thm}[\cite{deLa}]
\label{de La Harpe}
Let $X$ be a metric space that is geodesic and proper, let $G$ be a group and let $G \times X \rightarrow X$ be an action by isometries (say from the left).  Assume that the action is proper and that the quotient $G \backslash X$ is compact.  

Then the group $G$ is finitely generated and quasi-isometric to $X$.  More precisely, for any $x_0 \in X$, the mapping $G \rightarrow X$ given by $g \mapsto g x_0$ is a quasi-isometry.  
\end{Thm}

We will first show that $G \backslash \mathcal{MS}(C)$ is finite.

\begin{Lem}
\label{fundamental}
There are a finite number of  minimal genus Seifert surfaces in the core, called \emph{fundamental surfaces}, such that every Seifert surface for $K$ in the core is either fundamental or can be obtained by spinning a fundamental surface around some number of JSJ tori in the interior of the core.   In other words, there are a finite number of isotopy classes of minimal genus Seifert surfaces $\sigma_1, \ldots, \sigma_j$ in the core such that any other isotopy class of  minimal genus Seifert surfaces $\sigma' \in \mathcal{MS}(C)$ can be written as $\sigma' = g \sigma_k$ for some $g \in G$ and some fundamental surface $\sigma_k$. 
\end{Lem}

\begin{proof} 
Let $S$ be a minimal genus Seifert surface for the knot $K$ and assume that $S$ meets the JSJ tori $T_1, \ldots, T_n$ transversally and minimally.  Let $\mathcal N(T_i)$ be a neighborhood of $T_i$.  Recall that the blocks $B_k$ are the components of the complement of the JSJ tori.

For each $k$, each component of $S \cap B_k$ is isotopic to one of the finitely many possible properly embedded surfaces as given in Theorem \ref{fin surfaces}.  Inside each neighborhood  $\mathcal N(T_i)$, each component of $S \cap \mathcal N(T_i)$ is an incompressible annulus.

Therefore, the Seifert surface $S$ is obtained from some finite collection of the incompressible surfaces in each block by connecting these pieces with annuli across $T$.  Up to spinning around the torus, there are finitely many ways to connect the incompressible surfaces on either side of each torus.  Therefore, every Seifert surface $S$ is in the orbit of one of finitely many isotopy classes of Seifert surfaces coming from the finite number of ways of putting together the finite components in each block. 
\end{proof}

Next, we will show that the action is proper. 

\begin{Lem}
\label{proper}
The action of the group $G$ on $\mathcal{MS}(C)$ is proper.  That is, the stabilizer $G_v=\{ g \in G \ | \ g(v)=v \}$ is finite for every vertex $v$ in  $\mathcal{MS}(C)$.  
\end{Lem}

\begin{proof}  Let $v \in \mathcal{MS}(C)$ be a vertex of the core Kakimizu complex.   To prove that $G_v$ is finite, we will show that there is a monomorphism from $G_v$ to a finite group.

By Lemma~\ref{fundamental}, there are finitely many orbits $O_1,O_2, \ldots, O_r$ of the group $G$.   Choose a representative of each orbit $v_i \in O_i$. Let $V=\{v, v_1,...,v_r\}$, and let $d$ be the diameter of $V$. Let $B$ be a ball of diameter $d$ centered at the vertex $v$.  By construction, $V \subset B$.  Notice that each automorphism of $G_v$ preserves distance between vertices, so the ball $B$ is fixed set-wise.   Since $\mathcal{MS}(C)$ is locally finite by Claim~\ref{locallyfinite}, there are finitely many vertices in $B$.  This induces a homomorphism from the stabilizer $G_v$ to the permutation group of the (finitely many) vertices of $B$. 

To see that this homomorphism is injective, we note that the kernel consists of all elements of $G_v$ that fix $B$ point-wise.  Let $g$ be such an automorphism in the kernel.  Since $g$ fixes $B$ pointwise, then $g(v_i) = v_i$ for all $i$.  For any $x \in \mathcal{MS}(C)$,  $x = h(v_i)$ for some $i$, where $h$ is some element of $G$.  Since $G$ is abelian, $gh = hg$, so $g(x) = g(h(v_i)) = h(g(v_i)) = h(v_i) = x$.  Thus, $g$ fixes every $x \in \mathcal{MS}(C)$.  Because we quotiented out by the elements of $G'$ that act trivially, $g$ is the identity element in $G$ and in $G_v$.  Thus, the homomorphism from $G_v$ to the finite permutation group is injective, and thus $G_v$ is finite.  
\end{proof}

We can now combine these results to prove the following:

\begin{Lem} 
\label{locallyfiniteqi}
$\mathcal{MS}(C)$ is quasi-isometric to a a finitely generated abelian group.
\end{Lem}

\begin{proof}
As noted above, the metric on $\mathcal{MS}(C)$  is  the path metric, so the complex is properly geodesic.  Since each automorphism of $G$ takes disjoint surfaces to disjoint surfaces, it  preserves distances between vertices and thus acts isometrically on $\mathcal{MS}(C)$.  By Lemma~\ref{fundamental}, $G \backslash \mathcal{MS}(C)$ is finite and hence compact, and by Theorem~\ref{proper}, the action of $G$ on $\mathcal{MS}(C)$ via left multiplication is proper.  Thus Lemma~\ref{de La Harpe} implies that $\mathcal{MS}(C)$ is quasi-isometric to $G$, a finitely generated abelian group.
\end{proof}

\section{Proof of the Main Theorem}
\label{mainproofsect}

In the previous section, we proved that the core Kakimizu complex $\mathcal{MS}(C)$ is quasi-isometric to a finitely generated abelian group.  To prove our main result, it remains to show that for any knot $K$, $\mathcal{MS}(K)$ is quasi-isometric to $\mathcal{MS}(C)$.

\begin{Lem} 
\label{corelem}
The core Kakimizu complex $\mathcal{MS}(C)$ is quasi-isometric to the entire Kakimizu complex $\mathcal{MS}(K)$.  
\end{Lem}

\begin{proof}
We will show that $\mathcal{MS}(C)$ is quasi-isometric to $\mathcal{MS}(K)$ by showing that the embedding preserves distances and every vertex $\sigma$ in $\mathcal{MS}(K)$ is within a bounded distance from some vertex $\sigma'$ in $\mathcal{MS}(C)$ and the inclusion map $\mathcal{MS}(C) \hookrightarrow \mathcal{MS}(K)$ is quasi-isometric.  First we note that the proof of Lemma 3 of \cite{Sch} uses double curve sums to produce geodesics in $\mathcal{MS}(K)$ (as opposed to just paths). Since a double curve sum in $C$ produces a new surface in $C$, this implies that the geodesics between vertices of $\mathcal{MS}(C)$ constructed in this way will be contained in $\mathcal{MS}(C)$. Thus given two vertices in the core of the Kakimizu complex $\mathcal{MS}(C)$, measuring their distance in $\mathcal{MS}(C)$ is equivalent to measuring their distance in the entire Kakimizu complex $\mathcal{MS}(K)$. Thus the inclusion map preserves distances.

Let $S$ be a minimal genus Seifert surface for $K$ in the isotopy class $\sigma$.   If $S$ is contained in the core, then $\sigma \in \mathcal{MS}(C)$.  If not, then, $S$ must intersect a JSJ torus $T$ that bounds a block $B$ inside the core and $B'$ outside the core.  The torus $T$ separates the Seifert surface $S$ into a compact surface $S'$ inside the core and finitely many annuli $A_i$ outside the core.  In fact, there are at most $3g-2$ annuli $A_i$ since any minimal genus Seifert surface intersects the JSJ tori in at most $6g-4$ circles by Lemma \ref{bound}.  The boundaries $C_i^+ \cup C_i^-$ of each annulus $A_i$ can be rejoined by annuli $D_i$ that lie inside of $B$.  Attaching these $D_i$ to $S'$ yields a minimal genus Seifert surface $\widetilde{S}$ that lies completely in the core $C$ and is represented by an isotopy class $\sigma' \in \mathcal{MS}(C)$. 

Each time we surger the surface at a single annulus in this way, the new surface is disjoint from the previous surface. Thus $\sigma$ and $\sigma'$ are connected by a path in $\mathcal{MS}(K)$ of distance at most the number of annuli outside $C$, which is at most $3g - 2$. Thus, every $\sigma \in \mathcal{MS}(K)$ is within a bounded distance from $\mathcal{MS}(C)$ and thus the two complexes are quasi-isometric.
\end{proof}

\begin{proof}[Proof of Theorem~\ref{mainthm}]
In Lemma~\ref{corelem}, we showed that $\mathcal{MS}(K)$ is quasi-isometric to $\mathcal{MS}(C)$.  
Because quasi-isometry defines an equivalence relation, this implies, by Lemma~\ref{locallyfiniteqi}, that $\mathcal{MS}(K)$ is quasi-isometric to a finitely generated abelian group, and therefore quasi-isometric to  $\mathbb{Z}^n$ for some $n$.

Since $\mathcal{MS}(K)$ is quasi-isometric to $\mathcal{MS}(C)$, infinite diameter families of isotopy classes of Seifert surfaces can only be generated by spinning around tori in the interior of the core. For each block $B$ that is homeomorphic to the complement of a fibered link, spinning around the torus closest to the knot $K$ is equivalent (up to isotopy) to spinning around the remaining boundary tori. Thus $n$ is at most the number of JSJ tori in the interior of the core minus the number of fibered core blocks. Conversely, by Corollary 4.5 of~\cite{Ba2}, spinning around a torus that does not cobound a fibered block always produces a new Seifert surface. Thus, the value of the rank $n$ in $\mathbb Z^n$ is exactly the number of JSJ tori in the interior of the core minus the number of fibered blocks in the core.  
\end{proof}

\bibliographystyle{amsplain}
\bibliography{kakimizu}
 
\end{document}